\documentclass[12pt]{amsart}
\usepackage{amsmath,amssymb,amsthm,hyperref}
\usepackage[margin=1.25in]{geometry}
\title{Gaussian-weighted normal operators on Euclidean space}
\author{Yuzhou Joey Zou}
\address{Department of Mathematics and Statistics \\
Oakland University\\
Rochester, MI 48309, U.S.A.}
\email{yzou@oakland.edu}
\date{\today}
\newtheorem{theorem}{Theorem}[section]
\newtheorem{lemma}[theorem]{Lemma}
\newtheorem{corollary}[theorem]{Corollary}
\newtheorem{proposition}[theorem]{Proposition}

\theoremstyle{remark}

\newtheorem{remark}[theorem]{Remark}

\usepackage{xcolor}

\newcommand{\gpeuc}[1]{\mathcal{GP}_{\mathbb{R}^d}^{#1}}
\newcommand{\gpg}[1]{\mathcal{GP}_{\mathcal{G}}^{#1}}

\begin{document}

\sloppy

\begin{abstract}
We consider the normal operator of the X-ray transform, weighted with Gaussian weights, in Euclidean space $\mathbb{R}^d$ with $d\ge 3$. We show the eigenfunctions of the normal operator are joint eigenfunctions of the harmonic oscillator and the spherical Laplacian, and we relate the spectrum to that of elliptic operators in the 1-cusp pseudodifferential calculus.
\end{abstract}

\maketitle


\section{Introduction}
We consider Gaussian weighted normal operators of the form
\[N=e^{-\rho^2/2}I_0^\sharp e^{|p|^2} I_0 e^{-\rho^2/2},\]
where $I_0$ is the geodesic X-ray transform on $\mathbb{R}^d$ and $I_0^\sharp$ is the backprojection associated to the X-ray transform. We take the dimension $d$ to be at least $3$, although many of our results carry over to $d=2$ as well.

Here, $\rho$ is the distance to the origin in $\mathbb{R}^d$. On the data side, we parametrize geodesics by
\[\mathcal{G} := T\mathbb{S}^{d-1} = \{(v,p)\in\mathbb{S}^{d-1}\times\mathbb{R}^d\,:\,v\cdot p = 0\},\]
where $(v,p)$ corresponds to the geodesic starting at $p$ with velocity $v$. Note that $p$ is then the closest point on the geodesic to the origin, with $|p|$ representing the distance of the geodesic from the origin. In this parametrization, we have
\[I_0f(v,p) = \int_{\mathbb{R}} f(p+tv)\,dt\]
and
\[I_0^\sharp g(z) = \int_{\mathbb{S}^{d-1}} g(v,p(z,v))\,dv,\]
where $p(z,v) = z-(z\cdot v)v$ is the projection of $z$ onto $v^\perp$.

These exponentially weighted normal operators turn out to relate to the normal operators constructed by Vasy and Zachos \cite{VZ-24} in studying the X-ray transform on asymptotically conic spaces (including Euclidean space). They consider an asymptotically conic manifold, of dimension at least $3$, which comprises of a compact region glued to a conical region at infinity $(0,\epsilon)_x\times Y$ where the metric at infinity takes the form
\[g  = \frac{dx^2}{x^4} + \frac{h(x)}{x^2},\]
with $h(x)$ a family of metrics on $Y$ varying smoothly in $x$. The prototypical example is Euclidean space $\mathbb{R}^d$, with $x = 1/\rho$, $Y = \mathbb{S}^{d-1}$, and $h(x)=g_{\mathbb{S}^{d-1}}$, the usual round metric, independent of $x$. The authors considered the operator
\[e^{-\Phi}L\tilde\chi Ie^{\Phi},\]
where $\Phi = -\frac{1}{2x^2}$ (which equals $-\rho^2/2$ in the Euclidean case), $I$ is the geodesic X-ray transform,
$L$ is roughly the backprojection integrating uniformly (with respect to the Sasaki metric) along all unit tangent vectors, and $\tilde\chi$ is a fiberwise cutoff approximating a certain Gaussian. It can be shown (see Appendix \ref{sec:asymp-conic-normal}) that
\[\tilde\chi\approx e^{|p|^2-\rho^2},\]
and hence the operator of interest is, roughly speaking,
\[e^{\rho^2/2}I_0^\sharp e^{|p|^2-\rho^2} I_0e^{-\rho^2/2} = e^{-\rho^2/2}I_0^\sharp e^{|p|^2}I_0e^{-\rho^2/2}\]
where $I_0^\sharp e^{-\rho^2} = e^{-\rho^2}I_0^\sharp$ since $e^{-\rho^2}$ is independent of $v$.

In their paper, the authors showed (again, assuming $d\ge 3$) that $N$ is an elliptic pseudodifferential operator in the class $\Psi^{-1,-1}_{1c}$, the class of ``1-cusp'' pseudodifferential operators designed to provide parametrices to 1-cusp differential operators generated by $x^3\partial_x$ and $x\partial_y$. Thus $N^2$ should be the approximate inverse of some $\Psi$DO in $\Psi^{2,2}_{1c}$; note that a differential operator in this class roughly looks like
\[x^{-2}\left((x^3D_x)^2 + \sum_{j=1}^{n-1}(xD_{y_j})^2 + 1\right),\quad D = i^{-1}\partial.\]
The authors also considered a semiclassical version of this normal operator, which can give other exponentially-weighted normal operators to consider.


We now phrase our main results. Let $(\rho,\omega)\in[0,\infty)\times\mathbb{S}^{d-1}$ denote spherical coordinates in $\mathbb{R}^d$. For $k,l\in\mathbb{N}$, we consider functions of the form
\begin{equation}
\label{eq:phi}
\phi = e^{-\rho^2/2}L_k^{(l+\frac{d}{2}-1)}(\rho^2)\rho^lY_l(\omega)
\end{equation}
where $L_k^{(\alpha)}(x)$ is a generalized Laguerre polynomial (orthogonal with respect to the weight $x^\alpha e^{-x}$ on $(0,\infty)$), and $Y_l$ is a spherical harmonic satisfying
\[-\Delta_{\mathbb{S}^{d-1}}Y_l = l(l+d-2)Y_l,\]
where $-\Delta_{\mathbb{S}^{d-1}}$ is the \emph{spherical Laplacian}, i.e.\ the Laplace-Beltrami operator on $\mathbb{S}^{d-1}$. For $k,l\in\mathbb{N}$, let $V_{k,l}$ denote all functions of the form in \eqref{eq:phi}, i.e.
\begin{equation}
\label{eq:vkl}
V_{k,l} = \{e^{-\rho^2/2}L_k^{(l+\frac{d}{2}-1)}(\rho^2)\rho^lY_l(\omega)\,:\, -\Delta_{\mathbb{S}^{d-1}}Y_l = l(l+d-2)Y_l\}.
\end{equation}
These functions arise as the \emph{joint eigenfunctions} of the harmonic oscillator $-\Delta+\rho^2$ and the spherical Laplacian $-\Delta_{\mathbb{S}^{d-1}}$:
\begin{lemma}
\label{lem:eigenbasis}
For $\phi\in V_{k,l}$, we have that
\[(-\Delta+\rho^2)\phi = (4k+2l+d)\phi\quad\text{and}\quad -\Delta_{\mathbb{S}^{d-1}}\phi = l(l+d-2)\phi.\]
Moreover, by choosing an appropriate basis for the spherical harmonics $Y_l(\omega)$, one can form an orthogonal basis of $L^2(\mathbb{R}^d)$ comprising of functions in $V_{k,l}$ for $k,l\in\mathbb{N}$.
\end{lemma}
The computations are straightforward to check, and we prove this lemma in the Appendix for full rigor.

In phrasing the asymptotics for the eigenvalues of $N$, we will phrase them in terms of the quantities
\begin{equation}
\label{eq:Lambdakl}
\Lambda_{k,l}:= 4k+2l+d+l(l+d-2) = 4k+l^2+dl+d.
\end{equation}
These are precisely the eigenvalues of the sum $(-\Delta+\rho^2)+(-\Delta_{\mathbb{S}^{d-1}})$ of the harmonic oscillator and the spherical Laplacian, since 
\[((-\Delta+\rho^2)+(-\Delta_{\mathbb{S}^{d-1}}))\phi = (4k+2l+d+l(l+d-2))\phi\]
for $\phi\in V_{k,l}$. We remark that $(-\Delta+\rho^2)+(-\Delta_{\mathbb{S}^{d-1}})$ is indeed a (weighted) 1-cusp differential operator, elliptic in the class $\Psi_{1,c}^{2,2}$. Indeed, with $x = 1/\rho$, we have
\begin{align*}
-\Delta+\rho^2-\Delta_{\mathbb{S}^{d-1}} &= x^{-2}\left(-x^2(x^2\partial_x)^2-(d-1)x^3(x^2\partial_x)-x^4\Delta_{\mathbb{S}^{d-1}} + 1 - x^2\Delta_{\mathbb{S}^{d-1}}\right) \\
&=x^{-2}\left((x^3D_x)^2+ idx^2(x^3D_x)+ 1 + (1+x^2)(-x^2\Delta_{\mathbb{S}^{d-1}})\right),
\end{align*}
with $-x^2\Delta_{\mathbb{S}^{d-1}}$ an elliptic combination of $xD_{y_1},\dots,xD_{y_{d-1}}$, where $y_1,\dots,y_{d-1}$ are local coordinates on $\mathbb{S}^{d-1}$.

Our main results concern the eigendecomposition of the Gaussian-weighted normal operator $N$:
\begin{theorem}
\label{thm:eigenfunctions}
Let $k,l\in\mathbb{N}$. If $\phi\in V_{k,l}$, then
\[N\phi = \lambda_{k,l}\phi,\]
where $\lambda_{k,l}>0$ is given by the integral
\begin{equation}
\label{eq:lambda}
\lambda_{k,l} = \sqrt{\pi}\int_{\mathbb{S}^{d-1}}(1-v_1^2)^k(1-v_1^2+iv_1v_2)^l\,d\mathbb{S}^{d-1}(v).,
\end{equation}
where $v = (v_1,v_2,\dots,v_d)$ are the standard coordinates on $\mathbb{S}^{d-1}\subset\mathbb{R}^d$.
\end{theorem}
\begin{remark}
While the integral above involves a complex integrand, the integral is indeed real due to symmetry. The fact that the integral is \emph{positive} is less obvious; however, this ends up following since $N$ can in fact be written as $N = (I_0^w)^*I_0^w$, where $I_0^w = e^{|p|^2/2}I_0e^{-\rho^2/2}$ is injective on $L^2(\mathbb{R}^d)$ due to the injectivity of $I_0$ on suitably decaying functions on $\mathbb{R}^d$. 
See Section \ref{sec:algebra} for more details.
\end{remark}
In addition, we study the asymptotics of these eigenvalues:
\begin{theorem}
\label{thm:eigenvalue}
With $\Lambda_{k,l}$ and $\lambda_{k,l}$ defined in \eqref{eq:Lambdakl} and \eqref{eq:lambda}, there exist constants $c,C>0$, depending only on dimension, such that
\[c\Lambda_{k,l}^{-1/2}\le \lambda_{k,l}\le C\Lambda_{k,l}^{-1/2}\text{ for all }k,l\in\mathbb{N}.\]
\end{theorem}
Putting the two results together, we have a qualitative confirmation of the result in \cite{VZ-24}. Indeed, an elliptic operator in $\Psi_{1c}^{-1,-1}$ can morally be thought of as the $-1/2$ power of an elliptic operator in $\Psi_{1c}^{2,2}$, and the result shows that on exact Euclidean space, $N\in \Psi_{1c}^{-1,-1}$ is comparable with the elliptic operator $(-\Delta+\rho^2)+(-\Delta_{\mathbb{S}^{d-1}})\in\Psi_{1c}^{2,2}$, in that they have the same eigenfunctions, where the eigenvalues of $N$ are indeed roughly the eigenvalues of $(-\Delta+\rho^2)+(-\Delta_{\mathbb{S}^{d-1}})$ raised to $-1/2$.

We mention some related literature to the work in this article:
\begin{itemize}
\item The closest work to this article is by Davison \cite{D-81}, who studied Gaussian-weighted normal operators constructed from the \emph{Radon transform} (that is, replacing the X-ray transform $I_0$ by the Radon transform $\mathcal{R}$ integrating over hyperplanes instead of lines, and replacing the X-ray backprojection $I_0^\sharp$ by the Radon backprojection $\mathcal{R}^*$ averaging over all hyperplanes through a point), and derived a Singular Value Decomposition for such weighted normal operators. We also mention the work of Cnops \cite{C-96}, who also derived a SVD for weighted Radon transforms by relating them to Dirac operators, using a Clifford algebra-based approach.
\item Integral transforms with Gaussian weights have attracted practical attention, for example in the work of Nie et al.\ \cite{NLCH-21}, who gave an analytic inversion formula and a numerical inversion algorithm for a Gaussian-weighted Radon transform in 2 dimensions; note that the Radon transform is equivalent to the X-ray transform in 2 dimensions.
\item The usage of intertwining differential operators is based on the work of Maass \cite{M-91}. Similar techniques were used in \cite{M-20,MMZ-23,EMZ-26} to derive Singular Value Decompositions for weighted X-ray transforms in the Euclidean and hyperbolic disks.
\item Finally, the upshot of this paper is to connect the spectrum of the weighted normal operator $N$ to an elliptic operator in the 1-cusp calculus. While the calculus was initially developed in \cite{VZ-24} for the purpose of inverting operators like $N$ or tensorial analogues \cite{JV-24}, it has seen recent usage \cite{HJ-26,HJ-26-pre} in describing appropriate function spaces to study the nonlinear Schr\"odinger equation or to describe scattering for Schr\"odinger operators on curved spaces.
\end{itemize}
The article is organized as follows. In Section \ref{sec:algebra}, we describe the algebraic properties of the normal operator, such as basic mapping properties (Section \ref{subsec:mapping}) and intertwining relations with the harmonic oscillator and spherical Laplacian (Section \ref{subsec:mapping}). We then prove Theorem \ref{thm:eigenfunctions} in Section \ref{subsec:eigen}, showing that the eigenfunctions are those of the form in \eqref{eq:phi}, with eigenvalues given by the integral formula \eqref{eq:lambda}. In Section \ref{sec:eigen-asymp}, we prove the asymptotics stated in Theorem \ref{thm:eigenvalue}. The article is supported by Appendix \ref{sec:computations}, which contain lemmas of supporting computations, and Appendix \ref{sec:asymp-conic-normal} justifying the connection of the normal operator $N$ studied in this article with the normal operator of interest in \cite{VZ-24}.

\subsection*{Acknowledgments}
The author would like to thank Fran\c{c}ois Monard and Andr\'as Vasy for helpful discussions related to this project.

\section{Algebraic properties of the normal operator}
\label{sec:algebra}

In this section, we establish algebraic properties of the normal operator, such as mapping and intertwining properties, along with proving Theorem \ref{thm:eigenfunctions}. We note in this section that all lemmas not directly followed by a proof are proven in Appendix \ref{sec:computations}.

\subsection{Mapping properties on Gaussian-weighted polynomials}
\label{subsec:mapping}

In this subsection, we prove some basic mapping properties of $N$.

We consider the Lebesgue measure $dz = \rho^{d-1}\,d\rho\,d\mathbb{S}^{d-1}(\omega)$ on $\mathbb{R}^d$, and the measure $dv\,dp$ on the data space, where $dv = d\mathbb{S}^{d-1}(v)$, and on each fiber $dp$ should be interpreted as the $(n-1)$-dimensional Hausdorff measure on the linear subspace $\{p\in\mathbb{R}^d\,:\,p\cdot v = 0\}$ of $\mathbb{R}^d$. Note that, with respect to the $L^2$ inner product defined by these measures, the operators $I_0$ and $I_0^\sharp$ are formally adjoint. That is, for $f:\mathbb{R}^d\to\mathbb{C}$ and $g:\mathcal{G}\to\mathbb{C}$ decaying sufficiently quickly, we have $\langle I_0f,g\rangle_{L^2(\mathcal{G}, dv\,dp)} = \langle f, I_0^\sharp g\rangle_{L^2(\mathbb{R}^d),dz}$; indeed both inner products equal
\[\int_{\mathbb{S}^{d-1}}\int_{v^\perp}\int_{\mathbb{R}} f(p+tv)\overline{g(v,p)}\,dt\,dp\,dv.\]
Define $I_0^w := e^{|p|^2/2}I_0e^{-\rho^2/2}$, initially thought of as an operator $C_c^\infty(\mathbb{R}^d)\to C_c^\infty(\mathcal{G})$.
\begin{lemma}
\label{lem:l2bounded}
$I_0^w$ is bounded $L^2(\mathbb{R}^d,dz)\to L^2(\mathcal{G},dv\,dp)$, with adjoint $(I_0^w)^*=e^{-\rho^2/2}I_0^\sharp e^{|p|^2/2}$.
\end{lemma}
Thus we can write the weighted normal operator $N$ as
\[N=(e^{-\rho^2/2}I_0^\sharp e^{|p|^2/2})(e^{|p|^2/2}I_0e^{-\rho^2/2}) = (I_0^w)^*(I_0^w).\]
It will also be useful to note the following algebraic property. For $n\in\mathbb{N}$, let
\[\gpeuc{n} = \{e^{-\rho^2/2}q(z)\,:\,q\text{ polynomial on }\mathbb{R}^d, \deg(q)\le n\}\]
be the set of Gaussian-weighted polynomials of degree at most $n$ on $\mathbb{R}^d$, and
\[\gpg{n} = \text{span }\{e^{-|p|^2/2}q(p)r(v)\,:\,q,r\text{ polynomials on }\mathbb{R}^d, \deg(q)\le n\}\]
be the corresponding set of Gaussian-weighted polynomials on $\mathcal{G}$ where the $p$-dependence has degree at most $n$, where for $q,r$ polynomials on $\mathbb{R}^d$ the product $q(p)r(v)$ should be interpreted as the restriction to 
$\mathcal{G}$, viewed as the subset of $\mathbb{R}^d_p\times\mathbb{R}^d_v$ where $|v|^2=1$ and $v\cdot p = 0$.
We also write $\gpeuc{} = \bigcup_{n\in\mathbb{N}}\gpeuc{n}$ as the collection of all Gaussian-weighted polynomials on $\mathbb{R}^d$.

Recalling the spaces $V_{k,l}$ defined in \eqref{eq:vkl}, we have the following:
\begin{lemma}
\label{lem:gpeuc}
For each $n$, we have
\[\gpeuc{n} = \bigoplus_{2k+l\le n} V_{k,l}.\]
Moreover, $\gpeuc{}$ is dense in $L^2(\mathbb{R}^d,dz)$.
\end{lemma}

We now note the following mapping properties on the spaces $\gpeuc{n}$:

%
\begin{lemma}
\label{lem:i0w-map}
Let $q(z)$ be a polynomial on $\mathbb{R}^d$. Then
\[I_0^w(e^{-\rho^2/2}q) = e^{-|p|^2/2}(\sqrt{\pi}q(p) + q_r(v,p)),\]
where $q_r(v,p)$ is a sum of polynomials of the form  $\tilde{q}(p)\tilde{r}(v)$ with $\deg\tilde{q}<\deg q$. Consequently, $I_0^w$ maps $\gpeuc{n}$ to $\gpg{n}$ for all $n\in\mathbb{N}$.
\end{lemma}
\begin{proof} From \eqref{eq:i0wf}, we write
\begin{align*}
[I_0^w(e^{-\rho^2/2}q)](v,p) &= \int_{\mathbb{R}}e^{-t^2/2}\left(e^{-(|p|^2+t^2)//2}q(p+tv)\right)\,dt \\
&= e^{-|p|^2/2}\int_{\mathbb{R}} e^{-t^2}q(p+tv)\,dt \\
&= e^{-|p|^2/2}\left(\int_{\mathbb{R}} e^{-t^2}q(p)\,dt + \int_{\mathbb{R}}e^{-t^2}(q(p+tv)-q(p))\,dt\right).
\end{align*}
The first term is $\sqrt{\pi}q(p)$, while $q(p+tv)-q(p)$ is a sum of polynomials in $p$, of degree strictly less than $\deg q$, times polynomials in $tv$. Integrating in $t$ gives the result.
\end{proof}
We now note that if $\tilde{q}$ and $\tilde{r}$ are polynomials on $\mathbb{R}^d$, then
\begin{equation}
\label{eq:backproj-mapping}
\begin{aligned}
(I_0^w)^*(e^{-|p|^2/2}\tilde{q}(p)\tilde{r}(v))(z) &= e^{-\rho^2/2}\int_{\mathbb{S}^{d-1}} \left[e^{|p|^2/2} e^{-|p|^2/2}\tilde{q}(p)\tilde{r}(v)\right](z-z\cdot v,v)\,dv \\
&= e^{-\rho^2/2}q(z),\quad q(z) = \int_{\mathbb{S}^{d-1}} \tilde{q}(z-(z\cdot v)v)\tilde{r}(v)\,dv.
\end{aligned}
\end{equation}
Moreover, $\deg q\le\deg\tilde{q}$, as $\tilde{q}(z-(z-\cdot v))\tilde{r}(v)$ is a sum of terms which are polynomials of degree at most $\deg\tilde{q}$ times a polynomial in $v$, the latter of which are integrated to constants. Consequently, we immediately have:
\begin{lemma}
\label{lem:i0w*-map}
$(I_0^w)^*$ maps $\gpg{n}$ to $\gpeuc{n}$ for all $n\in\mathbb{N}$.
\end{lemma}
Combining Lemmas \ref{lem:i0w-map} and \ref{lem:i0w*-map}, as well as Equation \eqref{eq:backproj-mapping}, we obtain:
\begin{corollary}
\label{cor:N-leading}
For any $n\in\mathbb{N}$, we have that $N$ maps $\gpeuc{n}$ to $\gpeuc{n}$. Moreover, for any polynomial $q(z)$, we have
\[N(e^{-\rho^2/2}q(z)) = e^{-\rho^2/2}(q_0(z)+q_1(z)),\]
where $\deg q_1<\deg q$, and
\[q_0(z) = \sqrt{\pi}\int_{\mathbb{S}^{d-1}} q(z-(z\cdot v)v)\,dv.\]
\end{corollary}


\subsection{Intertwining Properties}
\label{subsec:intertwine}

In this subsection, we show $N$ commutes with the harmonic oscillator and the spherical Laplacian by finding explicit vector fields which intertwine weighted X-ray transforms and backprojections.

Given $A\in SO(n)$, let $R_A$ and $R_A^{\mathcal{G}}$ denote the actions on $\mathbb{R}^d$ and $\mathcal{G}$ defined by
\[R_A(z) = Az,\quad R_A^{\mathcal{G}}(v,p) = (Av,Ap).\]
Note that $R_A^{\mathcal{G}}$ does indeed define a function $\mathcal{G}\to\mathcal{G}$, since $|v|=1\iff |Av|=1$, and $v\cdot p = 0\iff (Av)\cdot (Ap) = 0$.

We note that pullbacks by $R_A$, $R_A^{\mathcal{G}}$ are intertwined through $I_0^w$ and $(I_0^w)^*$:
\begin{lemma}
\label{lem:rot-inter}
For $A\in SO(n)$, we have
\[(R_A)^*\circ(I_0^w)^* = (I_0^w)^*\circ(R_A^{\mathcal{G}})^*,\quad (R_A^{\mathcal{G}})^*\circ I_0^w = I_0^w\circ(R_A)^*,\]
where $(R_A)^*$ and $(R_A^{\mathcal{G}})^*$ are the pullbacks by $R_A$ and $R_A^{\mathcal{G}}$. Consequently,
\[(R_A)^*\circ N = N\circ (R_A)^*.\]
\end{lemma}
%
The fact that $N$ commutes with pullbacks by all rotations implies it commutes with the spherical Laplacian, which we show explicitly as follows:
\begin{corollary}
\label{cor:rot-commute}
If $B$ is any skew-symmetric matrix, then
\[((Bz)\cdot\nabla)\circ N = N \circ((Bz)\cdot\nabla).\]
In particular, $N$ commutes with all rotational vector fields of the form $z_i\partial_{z_j}-z_j\partial_{z_i}$.
\end{corollary}
\begin{proof}
This follows by noting that $\exp(tB)\in SO(n)$ for all $t\in\mathbb{R}$ if $B$ is skew-symmetric, and that for all smooth $f$ we have
\[\frac{d}{dt}\Big|_{t=0}\left[f(\exp(tB)z)\right] = \left(\frac{d}{dt}\Big|_{t=0}(\exp(tB)z)\right)\cdot\nabla f(z) = (Bz)\cdot\nabla f(z).\]
\end{proof}
Consequently, we obtain:
\begin{proposition}
$N$ commutes with the spherical Laplacian $-\Delta_{\mathbb{S}^{d-1}}$ on $\gpeuc{}$.
\end{proposition}
\begin{proof}
This follows from Corollary \ref{cor:rot-commute} by writing $-\Delta_{\mathbb{S}^{d-1}} = -\frac{1}{2}\sum_{i,j=1}^d(z_i\partial_{z_j}-z_j\partial_{z_i})^2$.
\end{proof}

We now aim to show that $N$ commutes with the harmonic oscillator. Define
\begin{equation}
\label{eq:pi}
P_i:=\partial_{p_i}-v_iv\cdot\partial_p,
\end{equation}
interpreted as a differential operator on $\mathcal{G}$. Interpreting the fiber of $\mathcal{G} = T\mathbb{S}^{d-1}$ over $v$ as $v^\perp\subset\mathbb{R}^d$, then the vector field $P_i$ is the projection of the vector field $\partial_{p_i}$ on $\mathbb{R}^d$ onto the subspace $v^\perp$. Note then that
\[\sum_{i=1}^d p_iP_i = \sum_{i=1}^d p_i\partial_{p_i} - p_iv_i(v\cdot\partial_p) = p\cdot\partial_p - (p\cdot v)(v\cdot\partial_p) = p\cdot\partial_p\]
since $p\cdot v = 0$. In addition, since $|v|^2=1$, we also have
\[\sum_{i=1}^d v_iP_i = \sum_{i=1}^d v_i\partial_{p_i} - \sum_{i=1}^d v_i^2v\cdot\partial_p = v\cdot\partial_p-1(v\cdot\partial_p) = 0\]
Finally, if we define $\Delta_p := \sum_{i=1}^d P_i^2$, then
\[\Delta_p = \sum_{i=1}^d (\partial_{p_i}-v_iv\cdot\partial_p)^2 = \sum_{i=1}^d \partial_{p_i}^2 - 2v_i\partial_{p_i}(v\cdot\partial_p) + v_i^2(v\cdot\partial_p)^2 = \Delta_{\mathbb{R}^d_p} - (v\cdot\partial_p)^2,\]
and since $v$ is a unit, normal vector for $v^\perp$ we see that the above is the Laplace-Beltrami operator on $v^\perp$, interpreted as a submanifold of $\mathbb{R}^d$.

We now establish the following intertwining relationships, valid on Schwarz spaces:
\begin{lemma}
\label{lem:pi-inter}
For the X-ray transform $I_0$, we have
\[P_i\circ I_0 = I_0\circ \partial_{z_i},\quad (p\cdot\partial_p)\circ I_0 = I_0\circ(\rho\partial_\rho+1),\]
and for the backprojection $I_0^\sharp$, we have
\[\partial_{z_i}\circ I_0^\sharp = I_0^\sharp\circ P_i,\quad \rho\partial_{\rho}\circ I_0^\sharp = I_0^\sharp\circ(p\cdot\partial_p).\]
\end{lemma}

With these relationships established, we prove (with the help of Lemma \ref{lem:gaussian-inter} from the Appendix):
\begin{proposition}
We have
\[(-\Delta+\rho^2)\circ N = N\circ (-\Delta+\rho^2)\]
when acting on $\gpeuc{}$.
\end{proposition}

\begin{proof}
From Lemma \ref{lem:gaussian-inter},
\[e^{\rho^2/2}\circ(-\Delta+\rho^2)\circ e^{-\rho^2/2} = -\Delta+2\rho\partial_{\rho} + d-\rho^2 + \rho^2 = -\Delta+2\rho\partial_\rho+d,\]
i.e.\ $(-\Delta+\rho^2)\circ e^{-\rho^2/2} = e^{-\rho^2/2}\circ(-\Delta+2\rho\partial_{\rho}+d)$.
Thus, if we try to intertwine $-\Delta+\rho^2$ through $I_0^w$, we have
\begin{equation}
\label{eq:osc-inter-1}
\begin{aligned}
(-\Delta+\rho^2)\circ(I_0^w)^* &= (-\Delta+\rho^2)\circ e^{-\rho^2/2}\circ I_0^\sharp\circ e^{|p|^2/2} \\
&= e^{-\rho^2/2}\circ(-\Delta+2\rho\partial_{\rho}+d)\circ I_0^\sharp \circ e^{|p|^2/2} \\
&= e^{-\rho^2/2}\circ I_0^\sharp\circ(-\Delta_p+2p\cdot\partial_p+d)\circ e^{|p|^2/2},
\end{aligned}
\end{equation}
where $-\Delta\circ I_0^\sharp = I_0^\sharp\circ -\Delta_p$ follows by writing $\Delta = \sum_{i=1}^d\partial_{z_i}^2$ and $\Delta_p = \sum_{i=1}^dP_i^2$. 

Next, we note that
\[e^{-|p|^2/2}\circ(-\Delta_p)\circ e^{|p|^2/2} = -\Delta_p-2p\cdot\partial_p-(d-1)-|p|^2\]
by applying Lemma \ref{lem:gaussian-inter}, interpreting $\Delta_p$ and $p\cdot\partial_p$ as the Laplacian and the dilation vector field on the $(d-1)$-dimensional Euclidean space $v^\perp\subset\mathbb{R}^d$. Lemma \ref{lem:gaussian-inter} also gives $e^{-|p|^2/2}\circ(p\cdot\partial_p)\cdot e^{|p|^2/2} = p\cdot\partial_p+|p|^2$, so
\begin{align*}
e^{-|p|^2/2}\circ(-\Delta_p+2p\cdot\partial_p+d)\circ e^{|p|^2/2} &= -\Delta_p-2p\cdot\partial_p-(d-1)-|p|^2 + 2(p\cdot\partial_p+|p|^2) + d \\
&=-\Delta_p + |p|^2 + 1.
\end{align*}
Thus, $(-\Delta_p+2p\cdot\partial_p+d)\circ e^{|p|^2/2}= e^{|p|^2/2}\circ(-\Delta_p + |p|^2 + 1)$, so \eqref{eq:osc-inter-1} gives
\begin{equation}
\label{eq:osc-inter-2}
(-\Delta+\rho^2)\circ (I_0^w)^* = (I_0^w)^*\circ (-\Delta_p+|p|^2+1).
\end{equation}
We now try to intertwine the last factor across $I_0^w$. To do so, we note similarly to above that
\[e^{-|p|^2/2}\circ (-\Delta_p+|p|^2+1)\circ e^{|p|^2/2} = -\Delta_p-2p\cdot\partial_p-(d-2),\]
so
\begin{align*}
(-\Delta_p+|p|^2+1)\circ e^{|p|^2/2}\circ I_0 &= e^{|p|^2/2}\circ(-\Delta_p-2p\cdot\partial_p-(d-2))\circ I_0 \\
&=e^{|p|^2/2}\circ I_0\circ(-\Delta - 2\rho\partial_\rho - d).
\end{align*}
Composing both sides with $e^{-\rho^2/2}$ thus gives 
\begin{align*}
(-\Delta_p+|p|^2+1)\circ I_0^w &= e^{|p|^2/2}\circ I_0\circ(-\Delta-2\rho\partial_\rho-d)\circ e^{-\rho^2/2} \\
&= I_0^w\circ e^{\rho^2/2}\circ(-\Delta-2\rho\partial_\rho-d)\circ e^{-\rho^2/2}.
\end{align*}
Furthermore,
\[e^{\rho^2/2}\circ(-\Delta-2\rho\partial_\rho-d)\circ e^{-\rho^2/2}=(-\Delta+2\rho\partial_\rho+d-\rho^2)-2(\rho\partial_\rho-\rho^2)-d = -\Delta+\rho^2.\]
Putting the two equations together thus gives
\begin{equation}
\label{eq:osc-inter-3}
(-\Delta_p+|p|^2+1)\circ I_0^w = I_0^w\circ(-\Delta+\rho^2).
\end{equation}
Combining \eqref{eq:osc-inter-2} and \eqref{eq:osc-inter-3} thus yields $(-\Delta+\rho^2)\circ N = N\circ(-\Delta+\rho^2)$, as desired.

\end{proof}

%


\subsection{Proof of Theorem \ref{thm:eigenfunctions}}
\label{subsec:eigen}
We now prove that the eigenfunctions of $N$ are functions of the form in \eqref{eq:phi}, with eigenvalues given by $\lambda_{k,l}$ defined in \eqref{eq:lambda}.
\begin{proof}[Proof of Theorem \ref{thm:eigenfunctions}]
We first claim that $N$ maps $V_{k,l}$ to itself, for any $k,l\in\mathbb{N}$. To see this,
let $n = 2k+l$, so that $V_{k,l}\subset\gpeuc{n}$. Then $N$, $-\Delta+\rho^2$, and $-\Delta_{\mathbb{S}^{d-1}}$ all act on $\gpeuc{n}$. Moreover,
\[(-\Delta+\rho^2)|_{V_{k,l}} = (4k+2l+d)\text{Id}|_{V_{k,l}},\quad (-\Delta_{\mathbb{S}^{d-1}})|_{V_{k,l}} = l(l+d-2)\text{Id}|_{V_{k,l}}.\]
Since $\gpeuc{n} = \bigoplus_{2k'+l'\le n}V_{k',l'}$, and $(k,l)\mapsto (4k+2l+d,l(l+d-2))$ is injective, it follows that $V_{k,l}$ is the maximal joint eigenspace in $\gpeuc{n}$ of $(-\Delta+\rho^2,-\Delta_{\mathbb{S}^{d-1}})$ with eigenvalues $(4k+2l+d,l(l+d-2))$.
Since $N$ commutes with $-\Delta+\rho^2$ and $-\Delta_{\mathbb{S}^{d-1}}$, it follows that $N$ must preserve the eigenspaces of $-\Delta+\rho^2$ and $-\Delta_{\mathbb{S}^{d-1}}$ acting on $\gpeuc{n}$. Hence, $N$ must map $V_{k,l}$ into itself.

We now note that the special orthogonal group $SO(d)$ acts on each $V_{k,l}$ via composition, i.e.\ $V_{k,l}$ is a representation of $SO(d)$. Moreover, since all functions in $V_{k,l}$ have the same radial part, the representation is isomorphic to the standard spherical harmonics representation of $SO(d)$, which is irreducible (see e.g.\ \cite{S-book} Section 2.3). It follows by Schur's lemma that $N|_{V_{k,l}} = \lambda_{k,l}\text{Id}|_{V_{k,l}}$ for some number $\lambda_{k,l}$, i.e. any $\phi\in V_{k,l}$ is an eigenfunction of $N$, with an eigenvalue $\lambda_{k,l}$ dependent only on $k$ and $l$ (and not the specific choice of $\phi\in V_{k,l}$). Moreover, since $N = (I_0^w)^*I_0^w$, with $I_0^w$ injective\footnote{This follows since $I_0^w = e^{|p|^2/2}I_0e^{-\rho^2/2}$, and $I_0$ is injective on functions of rapid decay.} on $L^2$, it follows that $\lambda_{k,l}>0$ for all $k,l$.

To compute $\lambda_{k,l}$, it suffices to compute $N\phi$ for \emph{any} nonzero $\phi\in V_{k,l}$, and see what multiple $\lambda_{k,l}$ satisfies $N\phi = \lambda_{k,l}\phi$. We will choose the spherical harmonic $Y_l$ to be $(z_1+iz_2)^l$ restricted to $\mathbb{S}^{d-1}$, in which case $\rho^lY_l(\omega) = (z_1+iz_2)^l$. Hence, the overall eigenfunction is
\[\phi = e^{-\rho^2/2}q(z),\quad\text{where }q(z) = L_k^{(l+d/2-1)}(|z|^2)(z_1+iz_2)^l.\]
Notice that $\deg q = n:= 2k+l$. Since $|z-(z\cdot v)v|^2 = |z|^2-(z\cdot v)^2$, Corollary \ref{cor:N-leading} gives that $\lambda_{k,l}\phi = N\phi = e^{-\rho^2/2}q_0(z)$ modulo $\gpeuc{n-1}$, where
\[q_0(z) = \sqrt{\pi}\int_{\mathbb{S}^{d-1}} L_k^{(l+d/2-1)}(|z|^2-(z\cdot v)^2)(z_1+iz_2 - (z\cdot v)(v_1+iv_2))^l\,dv.\]
It follows that $q_0(z) = \lambda_{k,l}q(z)$ modulo a polynomial of degree at most $n-1$. Moreover, if we renormalize $L_k^{(l+d/2-1)}(x)$ to have leading coefficient $1$, and we let
\begin{align*}
\tilde{q}(z) &= |z|^{2k}(z_1+iz_2)^l, \\
\tilde{q}_0(z) &= \sqrt{\pi}\int_{\mathbb{S}^{d-1}} (|z|^2-(z\cdot v)^2)^k(z_1+iz_2 - (z\cdot v)(v_1+iv_2))^l\,dv
\end{align*}
(i.e.\  we replace the $L_k^{(l+n/2-1)}(x)$ part in the definition of $q$ and $\tilde{q}$ by the polynomial $x^k$), then $q$ and $\tilde{q}$ have the same leading degree terms, and similarly for $q_0$ and $\tilde{q}_0$. Consequently, $q-\tilde{q}$ and $q_0-\tilde{q}_0$ are both degree at most $n-1$, so we also have $\tilde{q}_0(z) = \lambda_{k,l}\tilde{q}(z)$ modulo a polynomial of lower degree. Noting that both $\tilde{q}$, $\tilde{q}_0$ are homogeneous of degree $n$, it follows that we actually have the exact equality $\tilde{q}_0(z) = \lambda_{k,l}\tilde{q}(z)$. Finally, noting that $\tilde{q}(1,0,\dots,0) = 1$, it follows that $\tilde{q}_0(1,0,\dots,0) = \lambda_{k,l}\cdot 1$, and hence
\[
\lambda_{k,l} = \tilde{q}_0(1,0,\dots,0) = \sqrt{\pi}\int_{\mathbb{S}^{d-1}} (1-v_1^2)^k(1-v_1(v_1+iv_2))^l\,dv.
\]
This proves Theorem \ref{thm:eigenfunctions}.
\end{proof}

\section{Estimating the eigenvalues}
\label{sec:eigen-asymp}

In this section we study the asymptotics of
\[\lambda_{k,l} = \sqrt{\pi}\int_{\mathbb{S}^{d-1}}(1-v_1^2)^k(1-v_1^2+iv_1v_2)^l\,d\mathbb{S}^{d-1}(v)\]
where $k,l\in\mathbb{N}$, and $d\ge 3$. In this section, we will use the Big-Theta notation, where $f(x) = \Theta(g(x))$ if both $f(x)=O(g(x))$ and $g(x)=O(f(x))$. We aim to show that $\lambda_{k,l} = \Theta(\Lambda_{k,l}^{-1/2})$, where 
$\Lambda_{k,l}= 4k+2l+d+l(l+d-2)$.
Note that if $l^2\le k$, then $\Lambda_{k,l}=\Theta(k)$, while if $l^2\ge k$, then $\Lambda_{k,l}=\Theta(l^2)$. So, equivalently, we claim that if $l^2\le k$, then $\lambda_{k,l} = \Theta(k^{-1/2})$, while if $l^2\ge k$, then $\lambda_{k,l} = \Theta(l^{-1})$.

It turns out that terms of the form $k+l/2$ will appear frequently below, so we note that $k+l/2 = O(\Lambda_{k,l})$, while for $(k,l)\ne (0,0)$ we have $\Lambda_{k,l}^{1/2}=O(k+l/2)$. Consequently, for any $r>0$, we have $(k+l/2)^{-r} = O(\Lambda_{k,l}^{-r/2})$ for $k+l/2$ sufficiently large.

We go ahead and estimate the integral. We first note that
\[|1-v_1^2+iv_1v_2|^2 = (1-v_1^2)^2+v_1^2v_2^2\le (1-v_1^2)^2+v_1^2(1-v_1^2) = 1-v_1^2,\]
where the inequality follows from the fact that $v_2^2\le 1-v_1^2$ on the sphere. Hence, the integrand is bounded in absolute value by $(1-v_1^2)^{k+l/2}$. We note that if we parametrize $\mathbb{S}^{d-1}$, minus the two poles in the $v_1$ direction, by polyspherical coordinates $(v_1,w)\in(-1,1)\times\mathbb{S}^{d-2}$, via
\[(v_1,w)\mapsto\left(v_1,\sqrt{1-v_1^2}w\right),\]
then
\[d\mathbb{S}^{d-1}(v) = (1-v_1^2)^{(d-3)/2}\,dv_1\,d\mathbb{S}^{d-2}(w).\]

We go ahead and make this change of coordinates, writing $w = (w_2,w_3,\dots,w_n)$ with $w_2^2+\dots+w_n^2=1$. Then $v_2 = \sqrt{1-v_1^2}w_2$. The integral then becomes
\[\sqrt{\pi}\int_{\mathbb{S}^{d-2}}\int_{-1}^1 (1-v_1^2)^{k+(d-3)/2}\left(1-v_1^2+iv_1\sqrt{1-v_1^2}w_2\right)^l\,dv_1\,d\mathbb{S}^{d-2}(w).\]
We now try to estimate the inner integral, for a fixed $w\in\mathbb{S}^{d-2}$. First, we may restrict to $|v_1|\le\epsilon$ for a fixed small $\epsilon>0$, since the integrand is still bounded by $(1-v_1^2)^{k+(l+d-3)/2}$, which decays exponentially in $k+l/2$, and hence is $O(\Lambda_{k,l}^{-\infty})$ (i.e.\ $O(\Lambda_{k,l}^{-N})$ for all positive $N$), for $|v_1|>\epsilon$. Within this region, we can use the Taylor expansions $\sqrt{1-x^2} = 1-\frac{x^2}{2}+O(x^4)$ and $\log(1+x) = x-\frac{x^2}{2}+\frac{x^3}{3}+O(x^4)$ (with the $O(x^4)$ terms uniform on $|x|\le\epsilon$) to see that
\begin{align*}
    \log(1-v_1^2+iv_1\sqrt{1-v_1^2}w_2) &= iv_1v_2 - v_1^2\left(1-\frac{w_2^2}{2}\right) + iR_3(v_1,w_2) + R_4(v_1,w_2), \\
    \log(1-v_1^2) &=-v_1^2 + \tilde{R}_4(v_1),
\end{align*}
where $R_3(v_1,w_2) = v_1^3\left(\frac{w_2}{2}-\frac{w_2^3}{3}\right)$, and $R_4$ and $\tilde{R}_4$ are both $O(v_1^4)$, uniform in $|v_1|\le\epsilon$ and $|v_2|\le 1$. It then follows that
\begin{align*}
&\sqrt{\pi}(1-v_1^2)^{k+(d-3)/2}\left((1-v_1^2+iv_1\sqrt{1-v_1^2}w_2\right)^l \\
&= \sqrt{\pi}\exp\left(ilv_1w_2 - v_1^2\left(k+l\left(1-\frac{w_2^2}{2}\right)+\frac{d-3}{2}\right) + ilR_3+lR_4+k\tilde{R}_4\right)\\
&=\sqrt{\pi}\exp\left(\varphi_{k,l}(v_1,w_2)\right) + A, 
\end{align*}
where 
\[\varphi_{k,l}(v_1,w_2) = ilv_1w_2 - v_1^2(k+\tilde{l}(w_2)),\quad \tilde{l}(w_2) = l(1-w_2^2/2)+(d-3)/2\]
(note that $\tilde{l}(w_2)\ge l/2$ when $|w_2|\le 1$), and
\begin{equation}
\label{eq:A}
A = \sqrt{\pi}\exp\left(\varphi_{k,l}(v_1,w_2)\right)\left(\exp(ilR_3+lR_4+k\tilde{R}_4)-1\right). 
\end{equation}
We think of $\sqrt{\pi}\exp(\varphi_{k,l}(v_1,w_2))$ as the main term whose asymptotics can be computed explicitly by a Gaussian integral, and 
the other term $A$ as an error term. As such, we now write
\begin{equation}
\label{eq:integral-split}
\lambda_{k,l} = \int_{\mathbb{S}^{d-2}}\int_{-\epsilon}^{\epsilon} \sqrt{\pi}\exp(\varphi_{k,l}(v_1,w_2))\,dv_1\,d\mathbb{S}^{d-2}(w) +\int_{\mathbb{S}^{d-2}}\int_{-\epsilon}^\epsilon A\,dv_1\,d\mathbb{S}^{d-2}(w)
 + O(\Lambda_{k,l}^{-\infty}),
\end{equation}
and we aim to estimate each of the above integrals.

For the main term, we have:
\begin{lemma}
\label{lem:int-main}
\[\int_{\mathbb{S}^{d-2}}\int_{-\epsilon}^{\epsilon} \sqrt{\pi}\exp(\varphi_{k,l}(v_1,w_2))\,dv_1\,d\mathbb{S}^{d-2}(w) = \Theta(\Lambda_{k,l}^{-1/2}).\]
\end{lemma}
\begin{proof}
We first integrate in $v_1$, noting that we can change the limits of integration to $\mathbb{R}$ at exponentially small cost in $k+l/2$, to get
\begin{align*}
\int_{-\epsilon}^\epsilon \sqrt{\pi}\exp\left(\varphi_{k,l}(v_1,w_2)\right)\,dv_1 &= \int_{\mathbb{R}} \sqrt{\pi}\exp\left(ilv_1w_2 - v_1^2(k+\tilde{l}(w_2))\right)\,dv_1 + O(e^{-c(k+l/2)}) \\
&=\frac{\pi}{\sqrt{k+\tilde{l}(w_2)}}\exp\left(-\frac{l^2w_2^2}{4(k+\tilde{l}(w_2))}\right) + O(\Lambda_{k,l}^{-\infty})
\end{align*}
where $c = -\log(1-\epsilon^2)$.
We proceed to estimate the integral of the above expression, over $w\in\mathbb{S}^{d-2}$, and we split into the cases where $l^2\le k$ or $l^2>k$.
In the case $l^2\le k$, we have 
\[\frac{\pi}{\sqrt{k+\tilde{l}(w_2)}} = \Theta(\Lambda_{k,l}^{-1/2}),\quad \exp\left(-\frac{l^2w_2^2}{4(k+\tilde{l}(w_2))}\right) = \Theta(1).\]
It follows that the integral in $v_1$ is $\Theta(\Lambda_{k,l}^{-1/2})$ for all $w_2$, and hence
\begin{align*}
&\int_{\mathbb{S}^{d-2}}\int_{-\epsilon}^\epsilon \sqrt{\pi}\exp\left(\varphi_{k,l}(v_1,w_2)\right)\,dv_1\,d\mathbb{S}^{d-2}(w) \\
&=\int_{\mathbb{S}^{d-2}}\frac{\pi}{\sqrt{k+\tilde{l}(w_2)}}\exp\left(-\frac{l^2w_2^2}{4(k+\tilde{l}(w_2))}\right)\,d\mathbb{S}^{d-2}(w) + O(\Lambda_{k,l}^{-\infty}) = \Theta(\Lambda_{k,l}^{-1/2})
\end{align*}
in the case $l^2\le k$.

We now consider the case $l^2>k$, in which case $\Lambda_{k,l}^{-1/2} = \Theta(l^{-1})$. Since $l/2\le \tilde{l}(w_2)\le l+O(1)$, we thus have $\frac{1}{k+\tilde{l}(w_2)} = \Theta\left(\frac{1}{k+l}\right)$. It follows that there are $c_1,c_2>0$ such that
\[\exp\left(-c_1\frac{l^2w_2^2}{k+l}\right)\le \exp\left(-\frac{l^2w_2^2}{4(k+\tilde{l}(w_2))}\right) \le \exp\left(-c_2\frac{l^2w_2^2}{k+l}\right)\quad\text{ for all }|w_2|\le 1.\]
Moreover, $\frac{\pi}{\sqrt{k+\tilde{l}(w_2)}} = \Theta\left(\sqrt{\frac{1}{k+l}}\right)$. Thus, we have the inequality
\[\frac{c_1'}{\sqrt{k+l}}\exp\left(-c_1\frac{l^2w_2^2}{k+l}\right)\le \frac{\pi}{\sqrt{k+\tilde{l}(w_2)}}\exp\left(-\frac{l^2w_2^2}{4(k+\tilde{l}(w_2))}\right) \le \frac{c_2'}{\sqrt{k+l}}\exp\left(-c_2\frac{l^2w_2^2}{k+l}\right)\]
for some $c_1,c_1',c_2,c_2'>0$. By Laplace's Method, for any fixed $c>0$ we have
\begin{align*}
\int_{\mathbb{S}^{d-2}} \frac{1}{\sqrt{k+l}}\exp\left(-c\frac{l^2w_2^2}{k+l}\right)\,d\mathbb{S}^{d-2}(w) &= \frac{|\mathbb{S}^{d-3}|}{\sqrt{k+l}}\int_{-1}^1 (1-w_2^2)^{(d-4)/2}\exp\left(-\frac{cl^2}{k+l}w_2^2\right)\,dw_2\\
&=\frac{1}{\sqrt{k+l}}\Theta\left(\sqrt{\frac{k+l}{cl^2}}\right) \\
&=\Theta\left(l^{-1}\right) = \Theta(\Lambda_{k,l}^{-1/2})
\end{align*}
since we are in the case where $l^2>k$. It follows that
\[\int_{\mathbb{S}^{d-2}}\frac{\pi}{\sqrt{k+\tilde{l}(w_2)}}\exp\left(-\frac{l^2w_2^2}{4(k+\tilde{l}(w_2))}\right)\,d\mathbb{S}^{d-2}(w) = \Theta(\Lambda_{k,l}^{-1/2}),\]
and hence
\begin{align*}
&\int_{\mathbb{S}^{d-2}}\int_{-\epsilon}^\epsilon \sqrt{\pi}\exp\left(\varphi_{k,l}(v_1,w_2)\right)\,dv_1\,d\mathbb{S}^{d-2}(w) \\
&=\int_{\mathbb{S}^{d-2}}\frac{\pi}{\sqrt{k+\tilde{l}(w_2)}}\exp\left(-\frac{l^2w_2^2}{4(k+\tilde{l}(w_2))}\right)\,d\mathbb{S}^{d-2}(w) = \Theta(\Lambda_{k,l}^{-1/2})
\end{align*}
in case $l^2>k$ as well.
\end{proof}
For the other term, we show
\begin{lemma}
\label{lem:int-err}
For $A$ defined in \eqref{eq:A}, we have
\[\int_{\mathbb{S}^{d-2}}\int_{-\epsilon}^\epsilon A\,dv_1\,d\mathbb{S}^{d-2}(w) = O(\Lambda_{k,l}^{-3/4}).\]
\end{lemma}
\begin{proof}
Noting that $|ilR_3+lR_4+k\tilde{R}_4| = O((k+l/2)|v_3|^3)$ uniformly in $w_2$, we write
\[\exp(ilR_3+lR_4+k\tilde{R}_4)-1 = ilR_3 + lR_4 + k\tilde{R}_4 + O((k+l/2)^2|v_1|^6).\]
Moreover, since $\partial_{v_1}\varphi_{k,l}(v_1,w_2) = \partial_{v_1}(ilv_1w_2-v_1^2(k+\tilde{l}(w_2))) = ilw_2-2(k+\tilde{l}(w_2))v_1$, it follows that
\begin{align*}
ilR_3 = ilw_2v_1^3\left(\frac{w_2}{2}-\frac{w_2^3}{3}\right) 
&= v_1^3\left(\frac{1}{2}-\frac{w_2^2}{3}\right)(\partial_{v_1}\varphi_{k,l} + 2(k+\tilde{l}(w_2))v_1) \\
&= v_1^3\left(\frac{1}{2}-\frac{w_2^2}{3}\right)\partial_{v_1}\varphi_{k,l} + O((k+l/2)v_1^4).
\end{align*}
Combining this with $lR_4+k\tilde{R}_4 = O((k+l/2)v_1^4)$, we get
\begin{align*}
A &= \sqrt{\pi}\exp\left(\varphi_{k,l}\right)\left(\exp(ilR_3+lR_4+k\tilde{R}_4)-1\right) \\
&= \sqrt{\pi}\partial_{v_1}(\exp(\varphi_{k,l}))v_1^3\left(\frac{1}{2}-\frac{w_2^2}{3}\right) +\exp(\varphi_{k,l})(O((k+l/2)v_1^4) + O((k+l/2)^2v_1^6)).
\end{align*}

To integrate the first term with respect to $v_1$, we note that
\begin{align*}
\int_{-\epsilon}^{\epsilon} \partial_{v_1}(\exp(\varphi_{k,l}))v_1^3\,dv_1 &= \left[\exp(\varphi_{k,l}(v_1,w_2))v_1^3\right]\Big|_{-\epsilon}^{\epsilon} - \int_{-\epsilon}^\epsilon 3v_1^2\exp(\varphi_{k,l}(v_1,w_2))\,dv_1 \\
&= O(\Lambda_{k,l}^{-\infty}) + O\left(\int_{\mathbb{R}}v_1^2\exp(-(k+l/2)v_1^2)\,dv_1\right) \\
&= O(\Lambda_{k,l}^{-\infty}) + O((k+l/2)^{-3/2}) = O(\Lambda_{k,l}^{-3/4}),
\end{align*}
with the second equality following from $|\exp(\varphi_{k,l})| = \exp(-(k+\tilde{l}(w_2))v_1^2)\le \exp(-(k+l/2)v_1^2)$ as $\tilde{l}(w_2)\ge l/2$, and the last equality following because ${O((k+l/2)^{-r})} = O(\Lambda_{k,l}^{-r/2})$ for $(k,l)\ne (0,0)$. Consequently,
\[\int_{\mathbb{S}^{d-2}}\int_{-\epsilon}^\epsilon \sqrt{\pi}\partial_{v_1}(\exp(\varphi_{k,l})(v_1,w_2))v_1^3\left(\frac{1}{2}-\frac{w_2^2}{3}\right)\,dv_1\,d\mathbb{S}^{d-2}(w) = O(\Lambda_{k,l}^{-3/4}).\]
For the other two terms, we note that
\begin{align*}
&\int_{-\epsilon}^{\epsilon}|\exp(\varphi_{k,l}(v_1,w_2))|(k+l/2)v_1^4\,dv_1\\
&\le (k+l/2)\int_{\mathbb{R}}v_1^4\exp(-(k+l/2)v_1^2)\,dv_1 = O((k+l/2)^{-3/2}) = O(\Lambda_{k,l}^{-3/4}),
\end{align*}
and
\begin{align*}
&\int_{-\epsilon}^{\epsilon}|\exp(\varphi_{k,l}(v_1,w_2))|(k+l/2)^2v_1^6\,dv_1\\
&\le (k+l/2)^2\int_{\mathbb{R}}v_1^6\exp(-(k+l/2)v_1^2)\,dv_1 = O((k+l/2)^{-3/2}) = O(\Lambda_{k,l}^{-3/4}).
\end{align*}
Putting the three terms together thus yields $\int_{\mathbb{S}^{d-2}}\int_{-\epsilon}^\epsilon A\,dv_1\,d\mathbb{S}^{d-2}(w) = O(\Lambda_{k,l}^{-3/4})$.
\end{proof}
\begin{proof}[Proof of Theorem \ref{thm:eigenvalue}]
Using \eqref{eq:integral-split} and Lemmas \ref{lem:int-main} and \ref{lem:int-err}, we have $\lambda_{k,l} = \Theta(\Lambda_{k,l}^{-1/2}) + O(\Lambda_{k,l}^{-3/4}) = \Theta(\Lambda_{k,l}^{-1/2})$, as claimed.
\end{proof}

\begin{remark}
For the case $d=2$, Theorem \ref{thm:eigenfunctions} still holds (in this case, the ``spherical harmonics'' are just the complex exponentials $e^{\pm il\theta}$ on the circle). However, the estimate in Theorem \ref{thm:eigenvalue} fails to hold. In this case, the formula for $\lambda_{k,l}$ (which still holds) reduces to
\[\lambda_{k,l} = \sqrt{\pi}\int_{\mathbb{S}^1}(1-v_1^2)^k(1-v_1^2+iv_1v_2)^l\,d\mathbb{S}^1(v) = \sqrt{\pi}\int_0^{2\pi}\cos^{2k+l}(\theta)e^{il\theta}\,d\theta = \frac{\pi^{3/2}}{2^{l-1}}\binom{2k+l}{k}.\]
In the extreme case where $k=0$, this integral is $\frac{\pi^{3/2}}{2^{l-1}}$, which decays superalgebraically relative to $\Lambda_{0,l}\approx l^2$.
\end{remark}

\begin{remark}
An alternative formula for $\lambda_{k,l}$, which gives a different approach to computing the eigenvalue asymptotics, is
\begin{align*}
\lambda_{k,l} &= \sqrt{\pi}|\mathbb{S}^{d-2}|\int_{-1}^1(1-v_1^2)^{k+(l+d-3)/2}\frac{C_l^{(d/2-1)}\left(\sqrt{1-v_1^2}\right)}{C_l^{(d/2-1)}(1)}\,dv_1 \\
&= \sqrt{\pi}|\mathbb{S}^{d-2}|\int_{-\pi/2}^{\pi/2}\cos^{2k+l+d-2}(\theta)\frac{C_l^{(d/2-1)}(\cos(\theta))}{C_l^{(d/2-1)}(1)}\,d\theta,
\end{align*}
where $C_l^{(\alpha)}(x)$ are the ultraspherical/Gegenbauer polynomials, orthogonal on $[-1,1]$ with respect to the measure $(1-x^2)^{\alpha-1/2}\,dx$.  This formula can be obtained, either by following the same proof as Theorem \ref{thm:eigenfunctions} except choosing the spherical harmonic $Y_l(\omega)$ to be the zonal harmonic with respect to $e_1 = (1,\dots,0)$, i.e. $Z^{(e_1)}_l(\omega) = C_l^{(d/2-1)}(\omega_1)$, or by taking the original equation \eqref{eq:lambda}, converting to spherical coordinates, and using the Laplace-type integral representation
\[\frac{C_l^{(\alpha+1/2)}(\cos(\theta))}{C_l^{(\alpha+1/2)}(1)} = \frac{\Gamma(\alpha+1)}{\sqrt{\pi}\Gamma(\alpha+1/2)}\int_0^\pi (\cos(\theta)+i\sin(\theta)\cos(\phi))^n(\sin(\phi))^{2\alpha}\,d\phi.\]
In this article we use the double integral in \eqref{eq:lambda} to make an explicit stationary phase/Fourier transform argument, but the author believes a careful analysis of the single-variable formula can get the same asymptotics for $\lambda_{k,l}$.
\end{remark}

\appendix

\section{Proofs of Lemmas}
\label{sec:computations}

In this section, we prove some lemmas appearing in the main text. We begin by stating a simple computation appearing in several places:
\begin{lemma}
\label{lem:gaussian-inter}
If $\rho$ denotes the radial variable in $\mathbb{R}^d$, $\Delta = \sum_{j=1}^d\partial_{z_j}^2$ the usual Laplacian on $\mathbb{R}^d$, and $\rho\partial_\rho=\sum_{j=1}^d z_j\partial_{z_j}$ the dilation vector field on $\mathbb{R}^d$, then
\begin{align*}
e^{\pm\rho^2/2}\circ (-\Delta)\circ e^{\mp\rho^2/2} &= -\Delta\pm 2\rho\partial_\rho\pm d-\rho^2, \\
e^{\pm\rho^2/2}\circ (\rho\partial_\rho)\circ e^{\mp\rho^2/2} &= \rho\partial_\rho\mp\rho^2.
\end{align*}
\end{lemma}
\begin{proof}
Both follow from the calculation
\[e^{\pm\rho^2/2}\circ\partial_{z_j}\circ e^{\mp\rho^2/2} = \partial_{z_j}\mp z_j.\]
For the Laplacian, we see that
\[e^{\pm\rho^2/2}\circ\partial_{z_j}^2\circ e^{\mp\rho^2/2} = (e^{\pm\rho^2/2}\circ\partial_{z_j}\circ e^{\mp\rho^2/2})^2 = (\partial_{z_j}\mp z_j)^2 = \partial_{z_j}^2 \mp 2z_j\partial_{z_j} \mp 1 + z_j^2,\]
and hence
\begin{align*}
e^{\pm\rho^2/2}\circ (-\Delta)\circ e^{\mp\rho^2/2} &= -\sum_{j=1}^de^{\pm\rho^2/2}\circ\partial_{z_j}^2\circ e^{\mp\rho^2/2} \\
&= -\sum_{j=1}^d\left(\partial_{z_j}^2\mp 2z_j\partial_{z_j}\mp 1 + z_j^2\right) = -\Delta\pm 2\rho\partial_\rho\pm d-\rho^2.
\end{align*}
For the dilation vector field, we similarly see that
\[e^{\pm\rho^2/2}\circ (\rho\partial_\rho) \circ e^{\mp\rho^2/2} = \sum_{j=1}^d z_j(\partial_{z_j}\mp z_j) = \rho\partial_\rho\mp\rho^2. \]
\end{proof}

We now prove Lemma \ref{lem:gpeuc}, which helps prove part of Lemma \ref{lem:eigenbasis}:
\begin{proof}[Proof of Lemma \ref{lem:gpeuc}]
Noting that $L_k^{(\alpha)}(\rho^2)$ is a polynomial of degree $2k$ for any $\alpha$, and that $\rho^lY_l(\omega)$ is a polynomial of degree $l$, we see that $V_{k,l}\subset \gpeuc{2k+l}$. In the other direction, we recall that the spherical harmonics $Y_l$ have the property that the corresponding \emph{solid} spherical harmonic $\rho^lY_l$ is harmonic, i.e.\ $\Delta(\rho^lY_l(\omega)) = 0$. Moreover, we recall the \emph{Gauss decomposition} for polynomials (cf.\ \cite{ABW-book}, Theorem 5.7), namely that every polynomial $q(z)$ of degree $n$ can be written as a sum $q(z) = \sum_{2k\le n}\rho^{2k}q_{n-2k}(z)$, where $q_{n-2k}(z)$ is a harmonic homogeneous polynomial of degree $n-2k$ and hence of the form $\rho^lY_l(\omega)$ for some spherical harmonic $Y_l$, with $l=n-2k$. Each power $x^k$ can be written as a linear combination $x^k = \sum_{k'=0}^kc_{kk'}L_{k'}^{(\alpha)}(x)$ of the Laguerre polynomials for any $\alpha$; consequently $\rho^{2k}$ can be written as $\rho^{2k} = \sum_{k'=0}^k c_{kk'}L_k^{(l+d/2-1)}(\rho^2)$ with $l=n-2k$. It follows that any element $e^{-\rho^2/2}q(z)$ of $\gpeuc{n}$ can be written in the form
\[e^{-\rho^2/2}q(z) = \sum_{2k\le n}\sum_{k'=0}^k e^{-\rho^2/2}L_{k'}^{(n-2k+\frac{d}{2}-1)}(\rho^2)\rho^{n-2k}(c_{kk'}Y_{n-2k}(\omega))\in\bigoplus_{2k+l\le n}V_{k,l}.\]
This shows that $\bigoplus_{2k+l\le n}V_{k,l} = \gpeuc{n}$.

To show $\gpeuc{}$ is dense in $L^2(\mathbb{R}^d,dz)$, we note,
for any $\zeta\in\mathbb{R}^d$, that $e^{-\rho^2/2}e^{i\zeta\cdot z} = \sum_{j=0}^\infty \frac{i^j}{j!}e^{-\rho^2/2}(\zeta\cdot z)^j$ is in $\overline{\gpeuc{}}^{L^2}$; hence any $f$ in $\left(\overline{\gpeuc{}}^{L^2}\right)^\perp$ must satisfy the equation $\mathcal{F}(e^{-\rho^2/2}f)(\zeta) = 0$ for all $\zeta$, where $\mathcal{F}$ is the Fourier transform, thus showing that $e^{-\rho^2/2}f$, and hence $f$, must be zero.
\end{proof}

We now prove Lemma \ref{lem:eigenbasis}.
We recall that the lemma concerns functions of the form
\[\phi = e^{-\rho^2/2}L_k^{(l+\frac{d}{2}-1)}(\rho^2)\rho^lY_l(\omega).\]
\begin{proof}[Proof of Lemma \ref{lem:eigenbasis}]
The equation $-\Delta_{\mathbb{S}^{d-1}}\phi = l(l+d-2)\phi$ follows by applying the spherical Laplacian to the spherical harmonic, so we focus on proving that $\phi$ is an eigenfunction of the harmonic oscillator $-\Delta+\rho^2$. This calculation is routine using spherical coordinates using that the Laguerre polynomials $L_k^{(\alpha)}(x)$ are solutions of the equation $xy''+(\alpha+1-x)y'+ky=0$; we give a self-contained proof here. Writing $\phi = e^{-\rho^2/2}\tilde\rho$, Lemma \ref{lem:gaussian-inter} gives
\[(-\Delta+\rho^2)e^{-\rho^2/2}\tilde\phi = e^{-\rho^2/2}(-\Delta+2\rho\partial_\rho+d)\tilde\phi.\]
We now write $\tilde\phi = f(\rho^2)\rho^lY_l(\omega)$. Note that $\Delta(\rho^lY_l(\omega)) = 0$, and $\rho\partial_\rho(\rho^lY_l(\omega)) = l\rho^lY_l(\omega)$, so that $\nabla(f(\rho^2))\cdot\nabla(\rho^lY_l(\omega)) = \partial_\rho(f(\rho^2))\partial_\rho(\rho^lY_l(\omega)) = \frac{l}{\rho}\partial_\rho(f(\rho^2))\rho^lY_l(\omega)$ since $f(\rho^2)$ varies only in the radial direction.
It follows that
\[-\Delta(f(\rho^2)\rho^lY_l(\omega)) = \left(-\Delta-\frac{2l}{\rho}\partial_\rho\right)[f(\rho^2)]\rho^lY_l(\omega),\]
by applying the product rule for the Laplacian $\Delta(fg) = f\Delta g+g\Delta f+2\nabla f\cdot\nabla g$. Moreover,
\[\rho\partial_\rho[f(\rho^2)\rho^lY_l(\omega)] = \left(\rho\partial_\rho+l\right)[f(\rho^2)]\rho^lY_l(\omega).\]
Hence,
\begin{equation}
\label{eq:hosc-eigen-1}
\begin{aligned}
(-\Delta+\rho^2)\left[e^{-\rho^2/2}f(\rho^2)\rho^lY_l(\omega)\right]=e^{-\rho^2/2}\rho^lY_l(\omega)\left(-\Delta-\frac{2l}{\rho}\partial_\rho+2\rho\partial_\rho+2l+d\right)[f(\rho^2)].
\end{aligned}
\end{equation}
Finally, using the formula for the Laplacian in spherical coordinates, we have
\[-\Delta-\frac{2l}{\rho}\partial_\rho+2\rho\partial_\rho+2l+d = -\partial_\rho^2-\left(\frac{2l+d-1}{\rho}-2\rho\right)\partial_\rho + (2l+d).\]
If we let $x = \rho^2$, in which case $\partial_\rho = 2\rho\partial_x = 2\sqrt{x}\partial_x$, we have
\begin{align*}-\partial_\rho^2-\left(\frac{2l+d-1}{\rho}-2\rho\right)\partial_\rho + (2l+d) &= -(2\sqrt{x}\partial_x)^2-\left(\frac{2l+d-1}{\sqrt{x}}-2\sqrt{x}\right)2\sqrt{x}\partial_x+(2l+d)\\
&=-4x\partial_x^2-2\partial_x-(4l+2d-2)\partial_x+4x\partial_x+(2l+d) \\
&=-4\left(x\partial_x^2+(l+d/2-x)\partial_x\right)+(2l+d).
\end{align*}
It follows that, if $f(x) = L_k^{(l+d/2-1)}(x)$, i.e.\ $xf''(x)+(l+d/2-x)f'(x) = -kf(x)$, we have
\begin{equation}
\label{eq:hosc-eigen-2}
\begin{aligned}
\left(-\Delta-\frac{2l}{\rho}\partial_\rho+2\rho\partial_\rho+2l+d\right)[f(\rho^2)] &= \left[\left(-4\left(x\partial_x^2+(l+d/2-x)\partial_x\right)+(2l+d)\right)f\right](\rho^2) \\
&=(4k+2l+d)f(\rho^2).
\end{aligned}
\end{equation}
Putting together \eqref{eq:hosc-eigen-1} and \eqref{eq:hosc-eigen-2} shows the harmonic oscillator eigenfunction property. The orthogonal basis statement then follows from Lemma \ref{lem:gpeuc}.
\end{proof}

We now prove the remaining lemmas in Section \ref{sec:algebra}: Lemma \ref{lem:l2bounded}, regarding $L^2$ boundedness of $I_0^w$ and $(I_0^w)^*$, Lemma \ref{lem:rot-inter} on intertwining $I_0^w$ and $(I_0^w)^*$ with pullbacks by rotations, and Lemma \ref{lem:pi-inter}, on intertwining $I_0$ and $I_0^\sharp$ with certain vector fields.
\begin{proof}[Proof of Lemma \ref{lem:l2bounded}]
We prove boundedness, since then the adjoint statement follows from $I_0$ and $I_0^\sharp$ being formally adjoint. For fixed $(p,v)\in\mathcal{G}$ we have
\begin{equation}
\label{eq:i0wf}
I_0^wf(p,v) = e^{|p|^2/2}\int_{\mathbb{R}} e^{-|p+tv|^2/2}f(p+tv)\,dt = \int_{\mathbb{R}} e^{-t^2/2}f(p+tv)\,dt,
\end{equation}
with the last equality following from $|p+tv|^2 = |p|^2+t^2$ as $p\cdot v = 0$. Consequently, by Cauchy-Schwarz we have
\[|I_0^wf(p,v)|^2\le \left(\int_{\mathbb{R}} e^{-t^2}\,dt\right)\left(\int_{\mathbb{R}}|f(p+tv)|^2\,dt\right) = \sqrt{\pi}\int_{\mathbb{R}}|f(p+tv)|^2\,dt,\]
and hence
\begin{align*}
\|I_0^wf\|_{L^2(\mathcal{G},dv\,dp)}^2 &\le \int_{\mathbb{S}^{d-1}}\int_{v^\perp}\sqrt{\pi}\int_{\mathbb{R}}|f(p+tv)|^2\,dt\,dp\,dv \\
&= \sqrt{\pi}\int_{\mathbb{S}^{d-1}}\left(\int_{\mathbb{R}^d}|f(z)|^2\,dz\right)\,dv = \sqrt{\pi}|\mathbb{S}^{d-1}|\|f\|_{L^2(\mathbb{R}^d,dz)}^2,
\end{align*}
where the second line follows from the change of coordinates $z = p+tv$, $(p,t)\in v^{\perp}\times\mathbb{R}$ for a fixed $v\in\mathbb{S}^{d-1}$.
\end{proof}
\begin{proof}[Proof of Lemma \ref{lem:rot-inter}]
It suffices to show the statements with $I_0^w$ and $(I_0^w)^*$ replaced by $I_0$ and $I_0^\sharp$, acting on $C^\infty$ functions with sufficient decay. Indeed, $I_0^w = e^{|p|^2/2}I_0e^{-\rho^2/2}$ and $(I_0^w)^* = e^{-\rho^2/2}I_0^\sharp e^{|p|^2/2}$ are compositions of $I_0$ and $I_0^\sharp$, where the multiplication operators $e^{\pm\rho^2}$ and $e^{\pm|p|^2}$ commute with $(R_A)^*$ and $(R_A^{\mathcal{G}})^*$.

Thus, we aim to show $(R_A)^*\circ I_0^\sharp = I_0^\sharp\circ(R_A^{\mathcal{G}})^*$. For any $g\in C^\infty(\mathcal{G})$, we have
\begin{align*}((R_A)^*\circ I_0^\sharp)g(z) &= \int_{\mathbb{S}^{d-1}} g(v,p(Az,v))\,dv \\
&\overset{v=Av'}{=}\int_{\mathbb{S}^{d-1}} g(Av',p(Az,Av'))\,dv' \\
&=\int_{\mathbb{S}^{d-1}}g(Av',Ap(z,v'))\,dv' = (I_0^\sharp\circ(R_A^{\mathcal{G}})^*)g(z).
\end{align*}
To show $(R_A^{\mathcal{G}})^*\circ I_0 = I_0\circ (R_A)^*$, we note, for $f\in C^\infty(\mathbb{R}^d)$ with sufficient decay, that
\[((R_A^{\mathcal{G}})^*\circ I_0)f(v,p) = I_0f(Av,ap) = \int_{\mathbb{R}} f(Ap+tAv)\,dt = I_0[f\circ R_A](v,p) = (I_0\circ R_A^*)f(v,p),\]
as desired.
\end{proof}
\begin{proof}[Proof of Lemma \ref{lem:pi-inter}]
Writing $I_0f(p,v) = \int_{\mathbb{R}} f(p+tv)\,dt$, the identity $P_i\circ I_0 = I_0\circ\partial_{z_i}$ follows by differentiating under the integral, noting that
\[(\partial_{p_i}-v_iv\cdot\partial_p)[f(p+tv)] = \partial_{z_i}f(p+tv) - v_iv\cdot\nabla f(p+tv),\]
and $\int_{\mathbb{R}}v\cdot\nabla f(p+tv)\,dt = \int_{\mathbb{R}}\frac{d}{dt}[f(p+tv)]\,dt = 0$. The equality $(p\cdot\partial_p)\circ I_0 = I_0\circ(\rho\partial_\rho+1)$ follows from the previous identity by writing $p\cdot\partial_p = \sum_{i=1}^d p_iP_i$, noting that
\begin{align*}
\sum_{i=1}^d p_i\partial_{z_i}f(p+tv) &= \left(\sum_{i=1}^d (p_i+tv_i)\partial_{z_i}f(p+tv)\right) - t\sum_{i=1}^dv_i\partial_{z_i}f(p+tv) \\
&= (\rho\partial_{\rho}f)(p+tv) - tv\cdot\nabla f(p+tv),
\end{align*}
with $-\int_{\mathbb{R}}tv\cdot\nabla f(p+tv)\,dt = -\int_{\mathbb{R}} t\frac{d}{dt}[f(p+tv)]\,dt = \int_{\mathbb{R}}f(p+tv) = I_0f(p,v)$ from integration by parts.

For the backprojection, the identity $\partial_{z_i}\circ I_0^\sharp$ follows from differentiating under the integral by writing $I_0^\sharp g(z) = {\int_{\mathbb{S}^{d-1}}g(v,z-(z\cdot v)v)\,dv}$, and the equality $\rho\partial_{\rho}\circ I_0^\sharp = I_0^\sharp\circ(p\cdot\partial_p)$ follows from the previous identity by writing $\rho\partial_{\rho} = \sum_{i=1}^d z_i\partial_{z_i}$, noting that
\begin{align*}
\sum_{i=1}^d z_iP_ig(v,z-(z\cdot v)v) &= \left[\sum_{i=1}^d p_iP_ig\right](v,p) + (z\cdot v)\left[\sum_{i=1}^d v_iP_ig\right](v,p) \\
&=(p\cdot\partial_p)g(v,p),
\end{align*}
with $p = z-(z\cdot v)v$, since $\sum_{i=1}^d p_iP_i = p\cdot\partial_p$ and $\sum_{i=1}^dv_iP_i = 0$.
\end{proof}

\section{The normal operator in the asymptotically conic setting}
\label{sec:asymp-conic-normal}
Here, we translate the operator considered in \cite{VZ-24} to the Euclidean setting. The authors considered the operator
\[e^{-\Phi}L\tilde\chi Ie^{\Phi},\]
where $\Phi = -\frac{1}{2x^2} = -\rho^2/2$,
\[Lw(x,y) = \int w(\gamma_{x,y;\lambda,\omega})\,d\lambda\,d\mathbb{S}^{d-1}(\omega),\]
where $\gamma_{x,y;\lambda,\omega}$ is the geodesic passing through $(x,y)$ with tangent vector a positive multiple of $\lambda(x\partial_x) + \omega\cdot\partial_y$, and
\[\tilde\chi = \chi(x,y,\lambda/x,\omega)\]
where $\chi(x,y,\hat\lambda,\omega)$ is compactly supported in $\hat\lambda$ and is sufficiently close, in the Schwartz topology, to the Gaussian $e^{(\hat\lambda)^2/(2\alpha(x,y,\lambda,\omega))}$, where $\alpha(x,y,\lambda,\omega)$ satisfies that if $(x(t),y(t))$ is a geodesic starting at $(x,y)$ with velocity a positive multiple of $\lambda x\partial_x+\omega\cdot\partial_y$, then
\[x(t) = x + x\lambda t + x\alpha t^2 + O(xt^3).\]
For the Euclidean case, we note that $x\partial_x$ and $\partial_y$ are not unit vectors, but $x^2\partial_x = -\partial_\rho$ is, while we can choose local coordinates on $\mathbb{S}^{d-1}$ so that $\{\partial_{y_1},\dots,\partial_{y_{n-1}}\}$ are orthonormal for the metric $h(0)=g_{\mathbb{S}^{d-1}}$ at $y=0$, in which case $\{x\partial_{y_i}\}$ are nearly orthonormal for the Euclidean metric near $y=0$. Hence, thinking of $\lambda$ as small, we note that $\lambda\approx -\cos(\theta)$, with $\theta$ the angle between the vector $v$ positively parallel to $\lambda(x\partial_x)+\omega\cdot\partial_y$ and the unit outer radial vector field $\partial_\rho$. Note then that
\[|p(z,v)|^2 = |z-(z\cdot v)v|^2 = |z|^2\sin^2(\theta) = \rho^2(1-\lambda^2).\]
It follows that if $\hat\lambda = \lambda/x$, then
\[(\hat\lambda)^2 = \lambda^2\rho^2 = \rho^2 - |p(z,v)|^2.\]
Moreover,
\[\rho(t)^2 = \rho^2 + 2\rho\cos(\theta)t + \cos^2(\theta)t^2\implies \rho(t) = \rho\sqrt{1-2x\lambda t+x^2t^2}\]
so
\[x(t) = \frac{1}{\rho(t)} = \frac{x}{\sqrt{1-2x\lambda t+x^2t^2}} = x\left(1 + x\lambda t + \left(-\frac{1}{2}+\frac{3}{2}\lambda^2\right)x^2t^2 + O(x^3t^3)\right).\]
Thus
\[\alpha(x,y,\lambda,\omega) = -\frac{1}{2}+\frac{3}{2}\lambda^2\approx -\frac{1}{2}\]
as $x\to 0$, assuming $\lambda = O(x)$. Thus,
\[e^{(\hat\lambda)^2/(2\alpha)}\approx e^{-(\rho^2-|p(z,v)|^2)} = e^{|p(z,v)|^2-\rho^2}.\]
Thus, we can roughly think of $e^{-\Phi}L\tilde\chi$ as acting like
\[e^{-\Phi}L\tilde\chi g = e^{\rho^2/2}\int_{\mathbb{S}^{d-1}} e^{-\rho^2}e^{|p(z,v)|^2}g(v,p(z,v))\,d\mathbb{S}^{d-1}(v) = e^{-\rho^2/2}I_0^\sharp\left(e^{|p|^2}g\right),\]
so the operator of interest is roughly
\[(e^{-\Phi}L\tilde\chi)Ie^{\Phi} = e^{-\rho^2/2}I_0^\sharp e^{|p|^2} I_0e^{-\rho^2/2}.\]
This justifies the choice of normal operator $N$ to study in this article.

\end{document}